\documentclass[12pt, twoside, reqno] {amsart}

\usepackage{amsfonts}
\usepackage{amssymb}
\usepackage{latexsym}
\usepackage{epsf,graphicx}
\usepackage{epsf,graphicx,latexsym,%
}

\newtheorem{example}{Example}[section]

\newcommand{\be} {\begin{eqnarray}}
\newcommand{\ee} {\end{eqnarray}}
\newcommand{\bep} {\begin{eqnarray*}}
\newcommand{\eep} {\end{eqnarray*}}

\newcommand {\Hol}{\mathop{\rm Hol}\nolimits}

\renewcommand {\Re}{\mathop{\rm Re}\nolimits}
\newcommand {\Ker}{\mathop{\rm Ker}\nolimits}

\newcommand {\Null}{\mathop{\rm Null}\nolimits}

\newcommand {\DD}{\mathcal{D}}
\newcommand {\Ff}{\mathcal{F}}
\newcommand {\A}{\mathcal{A}}

\newcommand{\R}{{\mathbb R}}

\newcommand{\C}{{\mathbb C}}

\newtheorem{remar}{Remark}[section]
\newtheorem{examp}{Example}[section]
\newtheorem{defin}{Definition}[section]
\newtheorem{corol}{Corollary}[section]
\newtheorem{propo}{Proposition}[section]
\newtheorem{theorem}{Theorem}[section]
\newtheorem{lemma}{Lemma}[section]
\newtheorem{remark}{Remark}[section]

\newcommand{\rema}{\begin{remar}\rm}
\newcommand{\erema}{$\blacktriangleright$\end{remar}}

\newcommand{\exa}{\begin{examp}\rm}
\newcommand{\eexa}{$\blacktriangleright$\end{examp}}

\def\lwvec(#1 #2){\linewd 0.1
           \lvec(#1 #2)
           \linewd 0.05}

\begin{document}

\title[Continuous and holomorphic semicocycles]{Continuous and holomorphic semicocycles in Banach spaces}

\author[M. Elin]{Mark Elin}

\address{Department of Mathematics,
         Ort Braude College,
         Karmiel 21982,
         Israel}

\email{mark$\_$elin@braude.ac.il}

\author[F. Jacobzon]{Fiana Jacobzon}

\address{Department of Mathematics,
         Ort Braude College,
         Karmiel 21982,
         Israel}

\email{fiana@braude.ac.il}

\author[G. Katriel]{Guy Katriel}

\address{Department of Mathematics,
         Ort Braude College,
         Karmiel 21982,
         Israel}

\email{katriel@braude.ac.il}

\begin{abstract}
	
We study some fundamental properties of semicocycles over semigroups of self-mappings of a domain in a Banach space.
We prove that any semicocycle over a jointly continuous  semigroup is
itself jointly continuous. For semicocycles over semigroups which have generator,
we establish a sufficient condition for differentiablity with respect to the time variable, and hence for the semicocycle to satisfy a linear evolution problem, giving rise to the notion of `generator' of a semicocycle. Bounds on the growth of a semicocycle with respect to the time variable are given in terms of this generator. 

Special consideration is given to the case of holomorphic semicocycles, for which we prove an exact correspondence between certain uniform continuity properties of a semicocyle
and boundedness properties of its generator.

%{\footnotesize Key words and phrases: semigroup of holomorphic
%mappings, semigroup of composition operators, infinitesimal
%generator, dissipative operator, function convex in one direction.

%2000 Mathematics Subject Classification: 30C45, 47H20}
\end{abstract}

\maketitle

%\centerline\today

\section{Introduction}\label{sect-intro}

The notions of (semi)-groups and (semi)-cocycles over (semi)-groups play a central role in the theory of dynamical systems (see, for example, \cite{Katok, Latus, K-R, B-P}). These objects consist of families of mappings defined on an appropriate space and satisfying familiar algebraic conditions.
Groups and semigroups solve autonomous dynamical systems. Solutions of certain non-autonomous systems are known to be semicocycles.
Ergodic theory of smooth dynamical systems builds around the study of the derivative cocycle associated either to a map or a flow (see \cite[Chapters 5 and 6]{B-P}) which is the fundamental solution of variational equations (see, for example, \cite[Chapter 6]{Latus}). Depending on context, one considers families parameterized by either  discrete or continuous time. Additional restrictions on the space and members of the families are made, according to the problem of interest.

In this work the space in which semigroups act is a domain in a real or complex Banach space and we concentrate on continuous-time processes. We consider semigroups consisting of either continuous or holomorphic self-mappings of that domain (the book \cite{R-S1} and references therein can be used as a good source for the state of the art on semigroups of holomorphic mappings). In turn, semicocycles over these semigroups consist of one-parameter families of mappings from the same domain to a Banach algebra.  If this algebra is a space of bounded linear operators, then the norm-continuity in our setting is stronger than the strong continuity often suggested. On the other hand, we do not assume local compactness of the base space (cf. \cite{Latus}). Moreover, we study semicocycles over semigroups, which is more general than the usual setting over groups (cf. \cite{B-P, K-R}).

Our aim is to examine certain basic questions regarding semicocycles in Banach spaces, which are analogous to questions that are
studied in the theory of semigroups of mappings.

Section~\ref{sect-prelim} is devoted to some preliminary results from the theory of semigroups. Our main results are contained in Sections~\ref{sect-cont}--\ref{Hol-semico}.

We first consider the question of continuity of a semicocycle. While our definition of the concept only imposes
the requirement that a semicocycle be continuous with respect to the spatial and the time variable separately, we show in Section~\ref{sect-cont} that in
fact semicocycles are automatically jointly continuous with respect to both variables. To do this we employ a result of a
similar nature for semigroups obtained by Chernoff and Marsden \cite{Ch-M}. We also consider stronger types of continuity of semicocycles, which are important in Section~\ref{Hol-semico}, when dealing with holomorphic semicocycles.

Our next aim is to study the question of differentiability of a semicocycle with respect to the time parameter. This question is important, because it is related to the problem of generating semicocycles by linear evolution equations; cf. \cite{Latus}.  Following the analogy with the semigroup theory, we refer to the derivative with respect to time of the semicocycle at time $t=0$, if it exists, as the `generator' of the semicocycle.
Although not every semicocycle is differentiable with respect to time (see Examples~\ref{examp12+} and \ref{examp12b}),
we prove, in Section \ref{sect-diff}, a sufficient condition implying that semicocycles which are $C^1$-smooth with respect to the spatial variable (and satisfy an additional technical assumption) are automatically differentiable with respect to the time parameter, and are hence generated by
an evolution equation. This is a generalization of Theorem 2.2(b) in \cite{EJK}, where the one-dimensional
holomorphic case was investigated.

We also study the growth of the norm of a semicocycle with respect to the time parameter, obtaining bounds on this growth in
terms of the semicocycle generator.

In Section \ref{Hol-semico} we consider the case of holomorphic semicocycles on a bounded domain.
In this case we can characterize the class of semicocycles which are differentiable with respect to the time parameter:
it is precisely the class of semicocycles which have the property of uniform joint continuity. We also show that
a holomorphic semicocycle has the stronger property of $T$-continuity if and only if its generator is bounded
on sets which are strictly inside the domain.

\section{Preliminary results on semigroups}\label{sect-prelim}
\setcounter{equation}{0}

We begin with some standard notations.
Denote by $X$ and $Y$ two Banach spaces endowed with the norms $\|\cdot\|_X$ and $\|\cdot\|_Y$, respectively. Let $\DD\subset X$ and $\Omega\subset Y$ be domains (connected open sets). The set of all mappings continuous (respectively, smooth) on $\DD$ and taking values in $\Omega$ is denoted by $C(\DD,\Omega)$ (respectively, $C^1(\DD,\Omega)$). By $C(\DD)$ (respectively, $C^1(\DD)$) we denote the set of all continuous (smooth) self-mappings of $\DD$.

We begin with some general definitions of different types of continuity for arbitrarily families of mappings $\left\{f_t\right\}_{t\ge0},\ f_t:\DD\to Y$, where  $\DD\subset X$ is a domain in a Banach space $X$, and $Y$ is a Banach space, and relations between them. Recall that a bounded subset $\mathcal{D}^* \subset\mathcal{D}$ is said to lie strictly inside $\DD$ if it is bounded away from the boundary $\partial\DD$, that is, $\displaystyle\inf_{x\in\mathcal{D}_1} {\rm dist}(x,\partial \mathcal{D})>0$.

\begin{defin}\label{def-cont}
The family $\left\{f_t\right\}_{t\ge0},\ f_t \in C(\DD,Y)$ is said to be
\begin{itemize}
\item {\bf{separately continuous}} if for every $x_0\in\DD$, $f_t(x_0)$
is continuous with respect to $t\in[0,\infty)$;

\item {\bf{jointly continuous}} (JC, for short) at
$(t_0,x_0)\in[0,\infty)\times\DD$ if
\[
\lim_{t\to t_0,x\to x_0} f_t(x) =f_{t_0}(x_0).
\]
It is called jointly continuous on $I\times\DD$ ($I\subset[0,\infty)$ an interval) if it is JC at every point
$(t_0,x_0)\in I\times\DD$;

\item {\bf{uniformly jointly continuous}} (UJC, for short) if for every $t_0\ge0$ and $x_0\in\DD$ there exists a neighborhood $U$ of $x_0$  such that $f_t(x)\to f_{t_0}(x)$ as $t\to t_0$, uniformly on $U$.

\item {\bf locally uniformly continuous} ($T$-continuous, for short) at $t=t_0$
if for every subset $\mathcal{D}^*$ strictly inside~$\DD$,
\[
\sup_{x\in\mathcal{D}^*}\|f_t(x)-f_{t_0}(x) \| _Y
\to0\quad\mbox{as}\quad t\to t_0.
\]
It is called $T$-continuous if it is $T$-continuous at every
$t_0\ge0$.
\end{itemize}
\end{defin}

We notice that for the case where $X$ is finite-dimensional, local uniform continuity coincides with  continuity uniform on compact subsets.

The following simple assertion is of general character (not involving algebraic structure) and can be proved using standard arguments.
\begin{propo}\label{th-cont}
Let  $\left\{f_t\right\}_{t\ge0}\subset C(\mathcal{D},Y)$.
\begin{itemize}
  \item [(i)] If the family $\left\{f_t\right\}_{t\ge0}$
  is $T$-continuous, then it is uniformly jointly continuous.

   \item [(ii)] If the family $\left\{f_t\right\}_{t\ge0}$
  is uniformly jointly continuous, then it is jointly continuous on $[0,\infty)\times \DD$.

  \item [(iii)] If $X$ is a finite-dimensional and $\left\{f_t\right\}_{t\ge0}$ is jointly continuous  on $[0,\infty)\times \DD$,
  then it is $T$-continuous.
\end{itemize}
\end{propo}

We  mention that while separate continuity and joint continuity are very standard notions,  uniform types of continuity (UJC and $T$-continuity) are more specialized ones. They are mainly useful for the case of families of holomorphic mappings due to the following assertion, which follows from the Cauchy inequality for derivatives (see, for example, \cite[Proposition 2.3]{R-S1}).
Recall that if Banach spaces $X$ and $Y$ are complex, a mapping $F:\DD\to Y,\ \DD\subset X$, is said to be holomorphic if it is Frech\'{e}t differentiable
 at each point $x\in \mathcal{D}$. By $\Hol(\DD,\Omega)$ we denote the set of all holomorphic mappings on $\mathcal{D}$ with values in $\Omega\subset Y$. By $\Hol(\DD)$ we denote the set of all holomorphic self-mappings of $\DD$.

\begin{lemma}\label{lem-T-hol}
If a family  $\left\{f_t\right\}_{t\ge0}\subset \Hol(\DD,Y)$  is uniformly jointly continuous (respectively, $T$-continuous), then the family of the Frech\'{e}t derivatives $\left\{{f_t}'\right\}_{t\ge0}\subset \Hol(\DD,L(X,Y))$ is uniformly jointly continuous (respectively, $T$-continuous).
\end{lemma}

One central goal of the theory of dynamic systems is the
study of one-parameter semigroups.

\begin{defin}\label{def-sg-hol}
A family $\mathcal{F} =\left\{F_t\right\}_{t\ge0}\subset C(\DD)$ is called a one-parameter continuous semigroup (semigroup, for short) on $\DD$ if the following properties hold

(i) $F_{t+s}=F_t \circ F_s$ for all $t,s\geq 0$;

(ii) For all $x\in \mathcal{D}$, $\displaystyle \lim_{t\to
0^+}F_t(x) =x.$
\end{defin}

Note that, by this definition, a semigroup is separately continuous. It turns on that the algebraic structure automatically implies
joint continuity for $t>0$. The following theorem was proved in \cite{Ch-M} (see also \cite{Ball}), in the general setting of semigroups on metric spaces.

\begin{theorem}\label{th-Ch-M-B}
Let $\mathcal{F}=\{F_t\}_{t\ge0}$ be a semigroup on a metric space
$Z$. Then $\Ff$ is jointly continuous on $(0,\infty)\times Z$.
Moreover, if $Z$ is locally compact then $\Ff$ is
jointly continuous on $[0,\infty)\times Z.$
\end{theorem}

A counterexample of a semigroup which is not jointly continuous on $[0,\infty)\times Z$  was given in \cite{Ch}. Also, it is not true in general that semigroups are UJC. For example, any semigroup of linear operators which is not uniformly continuous is not UJC.

We now recall some facts related to the differentiability of a semigroup with respect to its parameter~$t$.

\begin{defin}\label{def-sg-gen}
Let $\mathcal{F} =\left\{F_t\right\}_{t\ge0}\subset C(\DD)$ be a semigroup on $\DD$.
If the limit
  \begin{equation}\label{gener}
f(x)= \lim_{t\to 0^{+}}\frac{1}{t} \left[ F_t(x) -x\right]
\end{equation}
 exists for every $x \in \DD$, then we say that $\mathcal{F}$ is differentiable with respect to its parameter~$t$, and  the mapping $f : \DD \to X$ defined by \eqref{gener} is called the {\sl (infinitesimal) generator} of the semigroup~$\mathcal{F}.$
%
%{\bf to remove: $f\in C(\DD, X)$ {\bf{why would the mapping $f$ be in $C(\DD, X)$?}} defined by \eqref{gener} is called the {\sl (infinitesimal) generator} of the semigroup~$\mathcal{F}.$}
\end{defin}

Note that if the limit \eqref{gener} exists, then the semigroup $\mathcal{F}$ solves  the Cauchy problem

\begin{equation}  \label{nS1}
\left\{
\begin{array}{l}
\displaystyle
\frac{\partial u(t,x)}{\partial t}=f(u(t,x)) \vspace{2mm} \\
u(0,x)=x,%
\end{array}%
\right.
\end{equation}%
where we set $u(t,x)=F_{t}(x)$.

Although on the best of our knowledge, no differentiability criterion for general semigroups is known, for semigroups of holomorphic mappings the following fact was proven by Reich and Shoikhet (see \cite[Theorems 6.8--6.9]{R-S1}).

\begin{theorem}\label{th-RS}
  Let $\Ff\subset\Hol(\DD)$ be a semigroup on a bounded domain in a complex Banach space $X$. Then $\Ff$ is $T$-continuous if and only if for each $x\in\DD$ the limit in \eqref{gener} exists, uniformly on subsets strictly inside $\mathcal{D}$, and defines mapping $f\in\Hol(\DD,X)$ which is bounded on each subset strictly inside $\DD$.
\end{theorem}

Recall in this connection that one of the surprising features of the infinite-dimensional holomorphy is that the inclusion $f\in\Hol(\mathcal{D},Y)$ does not imply that $f$ is bounded on all subsets strictly inside $\DD$ (see \cite{IJM-SLL-84, Har, R-S1}).

The following notion concerns the flow's spread over finite times.

\begin{defin}
Let $\Ff=\{F_t(x)\}_{t\geq 0}$ be a semigroup on a domain $\DD$. We say that $\Ff$ acts strictly inside $\DD$ if for every subset $\DD^*$ strictly inside $\DD$ and every $t_0>0$ the set $\{F_t(x):\ x\in\DD^*,\ t\in[0,t_0] \}$ lies strictly inside $\DD$.
\end{defin}

In the case where $X$ is finite-dimensional, any semigroup acts strictly inside by compactness. On the other hand, we do not know  whether in the infinite-dimensional case every semigroup $\Ff\in C(\DD)$ acts strictly inside. Clearly, if $\DD$ is the union of $\Ff$-invariant subsets each one of which is strictly inside $\DD$, then $\Ff$  acts strictly inside.
Moreover, it can be shown that if a semigroup $\Ff$ consists of holomorphic self-mappings of a domain $\DD$ equipped with the hyperbolic metric (see \cite{HL-79, G-R} and \cite[Section 3.6]{R-S1}), then  $\Ff$ acts strictly inside.

\section{Semicocycles and their continuity}\label{sect-cont}
\setcounter{equation}{0}

The main object of study in this work is the notion of
{\it{semicocycle}}. Throughout, we assume that $\A$ is a complex unital Banach algebra with
the unity $1_A$ such that $\|1_\A\|_\A=1$.

\begin{defin}\label{semicocycle}
Let $\mathcal{F}=\{F_t\}_{t\ge0}\subset C(\DD)$ be a semigroup. The family $\left\{\Gamma_t\right\}_{t\ge0}$ of mappings $\Gamma_t:\DD\rightarrow \A$ is called a semicocycle over $\mathcal{F}$ if it satisfies the following:

\begin{itemize}
  \item [(a)] the chain rule:
$\Gamma_t(F_s(x))\Gamma_s(x) = \Gamma_{t+s}(x)$ for all $t,s\ge0$
and $x\in\DD$;

\item [(b)] $\displaystyle \lim_{t\to 0^+}\Gamma_t(x)= 1_\A$ for
every $x \in \DD$.
\end{itemize}
\end{defin}

The simplest example of a semicocycle over $\mathcal{F}\subset C^1(\DD)$ is given by the Frech\'{e}t
derivatives of the semigroup: $\Gamma_t(x)={F_t}'(x)$, where $\A=L(X)$. Another
simple example is one independent of~$x$. Such semicocycles  coincide with uniformly continuous semigroups $\{\Gamma_t(x)=e^{ta}\}_{t\ge0},\ a\in\A$. Additional examples of semicocycles are given
by $\Gamma_t(x)= M(F_t(x))M(x)^{-1}$, where $M\in C(\DD,\A_*)$.
Our interest in semicocycles is motivated, in part, by the following construction.

\begin{propo}\label{th-sg}
    Let $X$ and $Y$ be complex Banach spaces, $\A=L(Y)$ be the algebra of bounded linear operators on $Y$. Assume that $\Ff=\{F_t\}_{t\ge0}\subset C(\DD)$ is a semigroup on a domain $\DD\subset X$ and $\Gamma_t:\R^+\to C(\DD,\A)$. The family $\left\{\Gamma_t\right\}_{t\ge0}$ is a
    semicocycle over $\Ff$ if and only if the family ${\widetilde \Ff}
    =\left\{\widetilde {F}_t\right\}_{t\ge0}$ defined by
    $\widetilde{F}_t(x,y)= \left(F_t (x), \Gamma_t (x) y \right)$
    forms a one-parameter continuous semigroup on the domain
    $\DD\times Y$.
\end{propo}

The semigroup $\widetilde{\Ff}$ was studied in \cite{E-2011} as an extension operator for semigroups of holomorphic mappings. We see that such operators necessarily involve semicocycles. Note also that the extended semigroup $\mathcal{\widetilde F}$ is often referred to as a linear skew-product flow; see, for example \cite{Latus, B-P}.

\begin{proof}
    First, a straightforward calculation shows that
    \begin{eqnarray*}
        &&\left(\widetilde{F}_s\circ \widetilde{F}_t\right) (x,y)=
        \widetilde{F}_s\left(F_t(x),\Gamma_t(x)y \right) \\
        &=&\! \left(F_s(F_t(x)) , \Gamma_s(F_t(x)) \Gamma_t(x)y \right) \\
        &=& \!  \left(F_{s+t}(x) , \Gamma_s(F_t(x)) \Gamma_t(x)y \right),
    \end{eqnarray*}
    that is, $\mathcal{\widetilde F}$ satisfies the semigroup property
    if and only if $\left\{\Gamma_t,\ t\ge0\right\}$ satisfies the
    chain rule. In addition,
    \[
    \lim_{t\to0^+}\widetilde{F}_t(x,y)=\lim_{t\to0^+}\left(F_t(x),
    \Gamma_t(x)y \right) =\left(x, \displaystyle \lim_{t\to
        0^+}\Gamma_t(x) y \right).
    \]
    Hence, $\displaystyle \lim_{t\to 0^+}\Gamma_t(x)=1_\A$  if and only if $\lim\limits_{t\to0^+}
    \widetilde{F}_t(x,y)=(x,y)$.
\end{proof}

We now proceed with the study of relations between different types of continuity of semicocycles.  First we prove that semicocycles are {\it{always}} jointly continuous. We will see that while Theorem~\ref{th-Ch-M-B} implies the joint continuity of semicocycles for $t>0$, the joint continuity at $t=0$ will require an additional argument. Note that, due to the algebraic relations,  joint continuity of a semicocycle on $\{0\}\times\DD$ can be extended to $[0,\infty)\times\DD$ -- but the other direction is not so immediate.

\begin{theorem}\label{th-JC}
        Every  semicocycle $\left\{\Gamma_t\right\}_{t\ge 0}\subset C(\mathcal{D},\A)$ over a semigroup $\mathcal{F}\subset C(\mathcal{D})$ is jointly continuous on  $(0,\infty)\times\DD$. If $\Ff$  is jointly continuous on    ${[0,\infty)\times \mathcal{D}}$, then $\left\{\Gamma_t\right\}_{t\ge0}$ is jointly continuous on    ${[0,\infty)\times \mathcal{D}}$.
\end{theorem}

\begin{proof}
    First we show that $\left\{\Gamma_t\right\}_{t\ge0}$ is jointly continuous on ${(0,\infty)\times \mathcal{D}}$. Each element of the algebra $\A$ can be considered as the bounded linear operator of left multiplication on $\A$, hence one can say that $\Gamma_t(x)\in L(\A)$. Therefore the semigroup
$\mathcal{\widetilde F}$ constructed in Proposition~\ref{th-sg} acts also on the metric space $\DD\times \A$ and by Theorem~\ref{th-Ch-M-B} is therefore jointly continuous on $(0,\infty) \times \DD$. In addition, we have
\[
\widetilde{F}_t(x,a)-\widetilde{F}_{t_0}(x_0,a) =
\left(F_t(x)-F_{t_0}(x_0),
\left(\Gamma_t(x)-\Gamma_{t_0}(x_0)\right)a \right).
\]
Thus the joint continuity of  $\mathcal{\widetilde F}$ implies
that $\displaystyle\lim_{t\to t_0,x\to x_0} \Gamma_t(x)=
\Gamma_{t_0}(x_0)$.

    It remains to prove the joint continuity at points in $\{0\}\times \mathcal{D}$. We now fix
    $x_0\in \mathcal{D}$, and will prove the joint continuity at $(0,x_0)$.

    We fix $t_0>0$ so that $\Gamma_{t_0}(x_0)$ is invertible (this is possible since $\Gamma_t(x_0)\rightarrow \Gamma_0(x_0)=1_{\A}$ as
    $t_0\rightarrow 0$).

    Let us define, for the chosen $t_0$, the following function ${V:[0,\infty)\times \mathcal{D}\rightarrow \A}$:
    $$V(t,x)=\int_{t_0}^t \Gamma_s(x)ds.$$
    The joint continuity of $\left\{\Gamma_t\right\}_{t\ge0}$ on ${(0,\infty)\times \mathcal{D}}$ implies that $V$ is
    also jointly continuous on ${(0,\infty)\times \mathcal{D}}$.
    Noting that
    $$\lim_{t\rightarrow t_0}\frac{1}{t-t_0}V(t,x_0)=\Gamma_{t_0}(x_0),$$
    and using the assumption that $\Gamma_{t_0}(x_0)$ is invertible,
    we see that we can choose $t_1>0$ sufficiently close to $t_0$ so that $\frac{1}{t_1-t_0}V(t_1,x_0)$ is invertible,
    and hence $V(t_1,x_0)$ is also invertible. By continuity of $V$ we can also ensure that $V(t_1,x)$
    is invertible for $x$ in some neighborhood of $x_0$.

    Using the semicocycle property
    $\Gamma_{t+s}(x)=\Gamma_s(F_t(x))\Gamma_t(x)$, we get, for
    any~$t$,
    \begin{eqnarray*}%\label{id0}
    &&V(t_1,F_t(x))\Gamma_t(x)=\int_{t_0}^{t_1} \Gamma_s(F_t(x))
    \Gamma_t(x)ds  \nonumber \\
    &=&\int_{t_0}^{t_1} \Gamma_{t+s}(x)ds =\int_{t+t_0}^{t+t_1}
    \Gamma_{s}(x)ds
    \nonumber  \\
    &=&V(t+t_1,x)-V(t+t_0,x).
    \end{eqnarray*}
    Since $V(t_1,x)$ is invertible for $x$ in a neighborhood of $x_0$, and the semigroup is jointly continuous,
    we have that $V(t_1,F_t(x))$ is invertible for $(t,x)$ sufficiently close to $(0,x_0)$, so that
    we can write the above as
    \begin{equation}\label{kr}\Gamma_t(x) =V(t_1,F_t(x))^{-1}[ V(t+t_1,x)-V(t+t_0,x)].\end{equation}
    Now the joint continuity of $V$ in $(0,\infty)\times \mathcal{D}$, and the joint continuity of $\Ff$, imply that
    $$\lim_{(t,x)\rightarrow (0,x_0)}\Gamma_t(x) =\lim_{(t,x)\rightarrow (0,x_0)} V(t_1,F_t(x))^{-1}[ V(t+t_1,x)-V(t+t_0,x)]$$
    $$=V(t_1,x_0)^{-1}[ V(t_1,x_0)-V(t_0,x_0)]
    =1_{\A},$$
    and we have shown the joint continuity of $\Gamma_t(x)$ at $(0,x_0)$.
\end{proof}

In contrast with joint continuity, we now give examples showing that in infinite-dimensional spaces, even for semicocycles over linear  semigroups and the algebra $\A=\C$, uniform types of continuity (UJC and $T$-continuity) do {\it{not}} always hold.

\begin{example}\label{examp12+}
Let  $X=\ell^1$, the space of absolutely convergent sequences. So, any vector $x\in X$ can be
written as $x=\{x_k\}_{k\ge1}$ with $\|x\|:=\sum\limits_{k=1}^\infty |x_k|<\infty$.
Consider the linear semigroup $\mathcal{F}=\{F_t\}_{t\ge0}$ on the open unit ball $\DD$ of $X$ defined
by $F_t(x)=\{e^{-kt}x_k\}_{k\ge1}$. Denote $\Gamma_t(x)=\exp \left[\gamma_t(x)\right]$, where $\gamma_t(x)=\sum\limits_{k=1}^\infty x_k(1-e^{kt})$.
Let verify that $\{\Gamma_t\}_{t\ge0}$ is a semicocycle over $\mathcal{F}$. Indeed, this family is continuous, $\Gamma_0(x)=1$ for every $x\in X$, and
\begin{eqnarray*}
&&\Gamma_t(F_s(x))\Gamma_s(x) =\exp\left(\gamma_t(F_s(x)) +\gamma_s(x) \right)\\
&=& \exp\left[ \sum_{k=1}^\infty \left( e^{-ks}x_k\left(1-e^{-kt}\right) + x_k  \left( 1- e^{-ks}
\right) \right) \right] \\
&=& \exp\left[ \sum_{k=1}^\infty x_k \left( 1-  e^{-kt-ks} \right)\right] = \Gamma_{s+t}(x),
\end{eqnarray*}
that is, the chain rule is satisfied. So, this family forms a semicocycle. Take any point $x^{(0)}\in\DD$ and the points $x^{(n)}$ such that the sequence $x^{(n)}-x^{(0)}$ has only one coordinate different from zero, namely, ${x^{(n)}}_n - {x^{(0)}}_n=\epsilon>0$, so that all of these points are on the $\epsilon$-distance from $x^{(0)}$. Further, $\gamma_t(x^{(n)})=\gamma_t(x^{(0)})+\epsilon(1-e^{-nt})$.  We see that $\gamma_t(x)\to 0$ as $t\to0^+$ but not uniformly in any neighborhood of any point $x^{(0)}$. Therefore the family $\{\gamma_t\}_{t\ge0}$ (hence $\{\Gamma_t\}_{t\ge0}$) is not UJC. At the same time,  $\{\Gamma_t\}_{t\ge0}$ is a JC semicocycle by Theorem~\ref{th-JC}.
\end{example}

\begin{example}\label{examp12}
Let  $X=c_{0}$, the space of all  sequences tending to zero
equipped with the $\sup$-norm. Any vector $x\in X$ can be
written as $x=\{x_k\}_{k\ge1}$ with $\lim\limits_{k\to\infty}x_k=0$. Consider the linear semigroup
$\mathcal{F}=\{F_t\}_{t\ge0}$ on the open unit ball $\DD$ of $X$ defined by $F_t(x)=e^{-t}x$. Obviously, this semigroup is $T$-continuous.
Take $\rho\in(0,1)$ and denote  $\Gamma_t(x)=\exp\left[\gamma_t(x)\right]$, where $\gamma_t(x)=\sum\limits_{k=1}^\infty \left(\frac{x_k}{\rho}\right)^k \left( 1- e^{-tk} \right)$.

First, we check that $\{\Gamma_t\}_{t\ge0}$ is a semicocycle over $\mathcal{F}$. Indeed, this family is continuous and
$\Gamma_0(x)=1$ for every $x\in X$. We verify that the chain rule holds:
\begin{eqnarray*}
&&\Gamma_t(F_s(x))\Gamma_s(x) =\Gamma_t(e^{-s}x)\Gamma_s(x)\\
&=& \exp\left[ \sum_{k=1}^\infty \left( \left(\frac{e^{-s}x_k}
\rho\right)^k \left( 1- e^{-tk} \right) + \left(\frac{x_k}\rho
\right)^k \left( 1- e^{-sk}
\right) \right) \right] \\
&=& \exp\left[ \sum_{k=1}^\infty \left(\frac{x_k}\rho\right)^k
\left( e^{-sk}- e^{-tk-sk} +1- e^{-sk} \right)\right] =
\Gamma_{s+t}(x).
\end{eqnarray*}

Now we show  that $\{\Gamma_t\}_{t\ge0}$ is not uniformly
continuous on~$\{\|x\|\le\rho\}$. To this end, for each
natural $n$, we define a point $x^{(n)}$ with $\|x^{(n)}\|=\rho$~by
\[
{x^{(n)}}_{k}= \left\{
\begin{array}{l}
\rho,\quad 0\leq k\leq n, \vspace{2mm} \\
0,\quad k>n.
\end{array}\right.
\]
Then
\begin{eqnarray*}
\gamma_t(x^{(n)})=  \sum_{k=1}^n
\left(1-e^{-tk}\right) \ge n\left(1-e^{-t}\right).
\end{eqnarray*}
Therefore,
\begin{eqnarray*}
&&\Gamma_t(x^{(n)}) = \exp\left[\sum_{k=1}^n \left(1-e^{-tk}
\right)\right] \ge \exp\left( n\left(1-e^{-t}\right) \right),
\end{eqnarray*}
which tends to $1$ as $t\to0$ for each fixed $n$ but this convergence is not uniform with respect to $n$. Therefore, this semicocycle is not $T$-continuous, while its uniform joint continuity is obvious.
\end{example}

Considering the extended semigroup $\widetilde\Ff$ defined in
Proposition~\ref{th-sg}, the relationship between uniform joint continuity and $T$-continuity of a semicocycle and those of the extended semigroup defined by it follows easily from the definition:

\begin{propo}\label{propo-t-cont}
Let $\left\{\Gamma_t\right\}_{t\ge0}\subset C(\DD, \A)$ be a semicocycle over some semigroup $\Ff\subset C(\DD)$. The semigroup $\mathcal{\widetilde F}$ is uniformly jointly continuous (respectively, $T$-continuous)  if and only if both $\Ff$ and $\left\{\Gamma_t\right\}_{t\ge0}$ are.
\end{propo}

One might conjecture that uniform joint continuity (or $T$-continuity) of a semicocycle at ${t_0=0}$ will imply the same property for all $t_0\geq0$. Unfortunately, we do not know whether such a conclusion concerning the $T$-continuity is true in general.  We can show it only under some restrictions which are always satisfied when $X$ is finite-dimensional.

\begin{propo}\label{lem-LUC}
Let $\{\Gamma_t\}_{t\ge0}\subset C(\DD,\A)$ be a semicocycle over
a semigroup $\Ff\subset C(\DD)$.

(i) If the semicocycle is $UJC$ at $t=0$ then it is $UJC$.

(ii) Assume in addition that $\{F_t\}_{t\ge0}$ acts strictly inside $\DD$ and every mapping $\Gamma_t, t\ge0,$ is bounded on sets strictly inside $\DD$. Then if the semicocycle is $T$-continuous at $t=0$, then it is $T$-continuous.
\end{propo}
\begin{proof}
	
(i) Assume the semicocycle is UJC at $t=0$.  Fix $t_0>0,x_0\in \DD$.
By the chain rule,
\begin{eqnarray}\label{cr}
\left\| \Gamma_t(x)-\Gamma_{t_0}(x)\right\| = \left\|[\Gamma_{t-t_0}(F_{t_0}(x))-1_\A]\Gamma_{t_0}(x)\right\| \nonumber\\
\leq \| \Gamma_{t-t_0}(F_{t_0}(x))-1_\A \| \| \Gamma_{t_0}(x)\|.
\end{eqnarray}
By the UJC property at $t=0$, there is a neighborhood $U$ of $F_{t_0}(x_0)$ such that $\Gamma_{t-t_0}(x)\rightarrow 1_\A$ as
 $t\rightarrow 0^+$, uniformly for $x\in U$. By continuity of $F_{t_0}$ we can choose a neighborhood $V$ of $x_0$ such that
$F_{t_0}(V)\subset U$. Therefore $\Gamma_{t-t_0}(F_{t_0}(x))\rightarrow 1$ as $t\rightarrow 0^+$, uniformly for $x\in V$.
 By continuity of $\Gamma_{t_0}$, we can assume (perhaps by making $V$ smaller), that $\|\Gamma_{t_0}(x)\|$ is bounded
 in $V$. Therefore (\ref{cr}) implies that $\Gamma_t(x)\rightarrow \Gamma_{t_0}(x)$ as $t\rightarrow 0^+$, uniformly for $x\in V$.
To prove the same for $t\to {t_0}^-$, we use the invertibility of semicocycle values proved in Theorem~\ref{th-invert1} below. Then we write the chain rule in the form
\[
\Gamma_t(x)=\left(\Gamma_{t_0-t}(F_t(x))   \right)^{-1} \Gamma_{t_0}(x)
\]
and repeat the previous consideration.

(ii) Since $\Ff$  acts strictly inside $\DD$, for any set $\DD_1$ strictly inside $\DD$ there is a set $\DD_2$ strictly inside $\DD$ such that $F_{t_0}(x)\in \DD_2$ for all $x\in\DD_1$. Noticing that
$\Gamma_{t-t_0}(\tilde x)$ tends to $1_\A$ as $t\to {t_0}^+$ uniformly
on~$\DD_2,$ we conclude, using (\ref{cr}), that $\Gamma_t(x)\to\Gamma_{t_0}(x)$ as $t\to {t_0}^+$, uniformly on $\DD_1$. The case $t\to{t_0}^-$ can be considered as in the proof of assertion (i).
\end{proof}

We conclude this section by considering invertibility of the values $\Gamma_t(x)$ of semicocycles.
Observe that it follows from condition (b) of
Definition~\ref{semicocycle} that if
$\left\{\Gamma_t\right\}_{t\ge0}$ is a semicocycle, then for every
fixed $x\in\mathcal{D}$ there exists $t_0=t_0(x)$ such that $\Gamma_t(x)$
is invertible whenever $t\le t_0$. It turns out that the joint
continuity proved above implies that {\it all} semicocycle values
are invertible.

\begin{theorem}\label{th-invert1}
Let  $\left\{\Gamma_t\right\}_{t \ge 0}\subset C(\mathcal{D},\A)$ be a
semicocycle. Then each value $\Gamma_t(x),\ t\ge0, x\in\mathcal{D},$ is
invertible.
\end{theorem}
\begin{proof}
Fix $x_0 \in \mathcal{D}$. By the joint continuity of a semicocycle (Theorem~\ref{th-JC}) and
condition (b) of Definition~\ref{semicocycle}, there exists
$\delta>0$ such that $\Gamma_t(x)$ is invertible for all $0 \le t
<\delta$ and $x\in\mathcal{D}$, such that $\|x-x_0\|<\delta$ . Denote
$T=\sup\left\{\widetilde T: \Gamma_t(x_0)\in\A_*, 0\le
t<\widetilde T \right\}$.

Assume that $T$ is finite. Consider the point $\tilde
x_0=F_T(x_0)$. Since $\Gamma_0(F_T(\cdot))=1_\A$, there exists
$0<\tau$ such that $\Gamma_t(x)$ is invertible for every $0\le t
<\tau$ and $\|x-\tilde x_0\|<\tau$. There exists
$0<\varepsilon<\tau/2$ such that
$$\|F_{T-\varepsilon}(x_0)-\tilde x_0\|=\|F_{T-\varepsilon}(x_0)-F_T(x_0)\|< \tau.$$

By the chain rule we have
\[
\Gamma_{T+\varepsilon}(x_0)=\Gamma_{2\varepsilon}(\tilde x_0)
\Gamma_{T-\varepsilon}(x_0),\quad \tilde x_0=F_{T} (x_0).
\]
Therefore, $\Gamma_{T+\varepsilon}(x_0)\in\A_*$, which contradicts
the finiteness assumption for $T$. Therefore $\Gamma_t(x_0)$ is
invertible for all $t\ge0$. Since $x_0$ is arbitrary,  each
$\Gamma_t(x)\in\A_*$ for all $x\in\mathcal{D}$.
\end{proof}

\begin{remark}\label{lem-identity}
(a) For the case where $\Ff=\{F_t\}_{t\in \R}$ is a group, the above theorem enables to extend a semicocycle $\{\Gamma_t\}_{t \geq 0}$ over $\Ff$ to a cocycle $\{\Gamma_t\}_{t \in \R}$, by setting $\Gamma_{-t}(x)=\left( \Gamma_t (F_{-t}(x)) \right)^{-1}$.

(b) An assertion in a sense `converse' to Theorem~\ref{th-invert1} also holds. Namely, if a family $\left\{\Gamma_t\right\}_{t\ge0}$ satisfies the chain rule with respect to a semigroup  $\mathcal{F}=\{F_t\}_{t\ge0}$ and, for some $t_1>0$, $\Gamma_{t_1}(x)$ is invertible for all $x \in \DD$, then $\Gamma_0(x)=1_\A$ for all
$x\in\mathcal{D}$, so $\left\{\Gamma_t\right\}_{t\ge0}$ is a semicocycle over $\mathcal{F}$. Indeed, applying the chain rule, we have $\Gamma_{t_1}(x)= \Gamma_{t+0}(x)= \Gamma_{t_1}(x)\Gamma_0(x)$  for all $x\in\mathcal{D}$. Then the invertibility of $\Gamma_{t_1}(x)$ implies $\Gamma_0(x)=1_\A$.

(c) For semicocycles consisting of holomorphic mappings, a stronger assertion than (b) holds. If there exist
$x_1\in\mathcal{D}$ and $t_1\ge0$ such that the operator
$\Gamma_{t_1}(x_1)$ is invertible, then $\Gamma_0(x)=1_\A$ for all
$x\in\mathcal{D}$. Indeed, since $\Gamma_{t_1}\in \Hol(\mathcal{D},\A)$, there exists a neighborhood $U$ of $x_1$ such that the operator $\Gamma_{t_1}(x)$ is
invertible for all $x \in U$. As above, $\Gamma_0(x)=1_\A$ whenever $x\in U$. The assertion follows by the uniqueness theorem for holomorphic
mappings.
\end{remark}

\section{Generation of semicocycles}\label{sect-diff}
\setcounter{equation}{0}

In this section we prove that every semicocycle $\{\Gamma_t\}_{t\ge0} \subset C^1(\DD,\A)$ satisfying some mild assumptions is
differentiable with respect to $t$ and can be reproduced as the
unique solution to an evolution problem.

Since both semigroup elements and semicocycle elements depend on
the time parameter $t\ge0$ and the spatial variable $x\in X$, we distinguish the notations
for derivatives. Namely, we denote the derivative with respect to
the parameter $t$ by $\frac{d} {d t}$ while the
Frech\'{e}t derivative with respect to $x$ by ${F_t}'(x),\
{\Gamma_t}'(x)$ and so on.

We begin by recalling
some basic results about non-autonomous evolution problems (see, for example, \cite{Kr, Pa})
\begin{equation}\label{cauchy_1}
\left\{
\begin{array}{l}
\displaystyle \frac{d v(t)}{d t} =a(t)v(t) \vspace{2mm} \\
v(0)=1_\A,
\end{array}%
\right.
\end{equation}
where $a,v:\R^{+}\to \A$ ($\A$, as above, is a unital Banach algebra).

\begin{theorem}\label{th-sol-cauchy_1}
Let $a\in C(\R^{+}, \A)$. Then
\begin{itemize}
\item[(i)] the evolution problem~\eqref{cauchy_1} is equivalent to
the integral equation
\begin{equation}\label{integral-form}
v(t) =  1_\A + \int_0^t a(s)v(s) ds;
\end{equation}

\item[(ii)] the evolution problem~\eqref{cauchy_1} (hence, the
integral equation~\eqref{integral-form}) has a unique solution
$u\in C(\R^{+},\A)$. Moreover,
\begin{equation}\label{estim-u}
\|u(t)\|_\A\le \exp \left(\int_0^t \mu\left(a(s)\right)ds \right),
\end{equation}
where $\mu(a)=\lim \limits_{t \to
0^+}\displaystyle\frac{\|1_\A+ta\|_\A-1}{t}.$
\end{itemize}
\end{theorem}

In particular, assertion (i) and the existence and uniqueness statement in (ii) can be found in various monographs (see, for
example, \cite{M-S, Kr}), while estimate \eqref{estim-u} was proven in~\cite{AG1964}.

Our specific interest is in the special case of \eqref{cauchy_1}
in which $a(t)=B(F_t(x))$, where $\mathcal{F}=\{F_t\}_{t\ge0}
\subset C(\mathcal{D})$ is a semigroup and $B\in C(\mathcal{D},\A)$. In
this case it is natural to replace the function $v(t)$ above by a
mapping ${ v(t,x) : \R^+ \times \mathcal{D} \to \A.}$
By~Theorem~\ref{th-sol-cauchy_1}, the corresponding evolution
problem has a unique solution $u(t):= u(t,x)$ for every fixed
$x\in\mathcal{D}$. It turns out that in this special case the
solution is a semicocycle over the semigroup $\mathcal{F}$.

\begin{theorem}\label{th-sol-cauchy}
Let $B\in C(\mathcal{D},\A)$, and let $\mathcal{F}= \{F_t\}_{t\ge0}
\subset C(\mathcal{D})$ be a semigroup. Denote
by $u(t,x),\ t\ge0,\ x\in \mathcal{D},$ the unique solution to the evolution
problem
\begin{equation}\label{cauchy}
\left\{
\begin{array}{l}
\displaystyle \frac{d v(t,x)}{d t} =B(F_t(x))v(t,x) \vspace{2mm} \\
v(0,x)=1_\A.
\end{array}%
\right.
\end{equation}
Then the family $\{\Gamma_t(x):=u(t,x)\}_{t\ge0}$ is a semicocycle
over $\Ff$.
\end{theorem}

The proof is just a repetition of the proof of Theorem~ 2.2(a) in
\cite{EJK}.

\begin{corol}\label{cor-gen}
Let $f\in C(\mathcal{D},X)$ be a semigroup generator and $B\in
C(\mathcal{D}, \A)$. Then the mapping $\widetilde f$ defined by
$\widetilde f(x,a)=\left(f(x),B(x)a \right)$ generates a semigroup
$\Phi$ on $\mathcal{D}\times \A$. A point $(x_0,a_0)$ is
a stationary point of $\Phi$ if and only if $x_0\in\Null f$ and
$a_0\in \Ker B(x_0)$.
\end{corol}
In fact, the generated semigroup $\Phi$ in this corollary coincides with the extended semigroup (linear skew-product flow) $\widetilde\Ff$
described in Proposition~\ref{th-sg}. The proof follows automatically from Proposition~\ref{th-sg} and Theorem~\ref{th-sol-cauchy}.

\begin{example}\label{examp12a}
As in Example~\ref{examp12}, consider the space
${X=c_0}$, the linear semigroup $\mathcal{F}=\{e^{-t}\cdot\}_{t\ge0}$
and the semicocycle defined by
\[
\Gamma_t(x)=\exp\left[ \sum_{k=1}^\infty
\left(\frac{x_k}{\rho}\right)^k \left( 1- e^{-tk} \right)\right].
\]
It is easy to see that this  semicocycle is differentiable with respect to $t$ in spite of the fact that it is not $T$-continuous. Further, consider  the extended semigroup
$\widetilde{F}_t(x,a)=\left(F_t(x),\Gamma_t (x) a \right)$, which
acts on the domain $\DD\times\A$. Differentiating this semigroup,
we see that it is generated by the mapping
\[
\widetilde{f}(x,a)=\left(-x, \left(\sum_{k=1}^\infty k
\left(\frac{x_k}{\rho}\right)^k \right)a \right).
\]
Notice that by Proposition~\ref{propo-t-cont}, the extended semigroup is not $T$-continuous.  This example shows that on unbounded domains a semigroup of holomorphic mappings may be generated but not $T$-continuous (in contrast with the case of bounded domains; cf. Theorem~\ref{th-RS}).
\end{example}

The result of Theorem~\ref{th-sol-cauchy} naturally raises the
question as to whether {\it{every}} semicocycle is differentiable with respect to $t$ and can be produced
as the solution of a corresponding evolution problem. In general, the answer is negative, what follows from Example~\ref{examp12+}. In fact, even in the one-dimensional case there are non-differentiable semicocycles.

\begin{example}\label{examp12b}
  Consider the linear semigroup
$\Ff=\{e^{-t}\cdot\}_{t\ge0}$ on the interval $(-1,1)$. Denote
\[
\Gamma_t(x)=\left\{
\begin{array}{l}
e^{-t}, \qquad \qquad\mbox{if} \quad |x|<\frac12, \vspace{2mm} \\
2|x|e^{-t},\qquad\  \mbox{if} \quad \frac12\le |x|<\min\left \{1,\frac12 e^{t}\right\}, \vspace{2mm} \\
1\qquad \qquad\quad\mbox{otherwise}.
\end{array}
\right.
\]
A direct calculation shows that $\{\Gamma_t,\ t\ge0\}$ is a (real-valued) semicocycle over
$\mathcal{F}$. For every $x,\ |x|>\frac12$,  this semicocycle is not differentiable with respect to $t$ at the point $t=\ln(2|x|)$.
\end{example}

In Theorem~\ref{th-differentiability} below we establish sufficient conditions for the differentiability of smooth semicocycles. This assertion generalizes Theorem~2.2(b) in \cite{EJK} to infinite-dimensional spaces.

To obtain this result, we use the
mapping ${V:[0,\infty)\times \mathcal{D}\rightarrow \A}$ defined by
\begin{equation}\label{V-func}
V(t,x)=\int_{0}^t \Gamma_s(x)ds,
\end{equation}
which is similar to one used in the proof of Theorem~\ref{th-JC}. The properties of $V$ that will be needed below are given by

\begin{lemma}\label{lem-V}
Let $\{\Gamma_t\}_{t\ge0}\subset C(\mathcal{D},\A)$ be a
semicocycle over $\Ff.$ Then the following assertions hold:
\begin{itemize}
\item[(i)] If $\{\Gamma_t\}_{t\ge0}$ is jointly continuous on
$[0,\infty)\times \DD$ then the family $\{V(t,\cdot)\}_{t\geq 0}$
is also jointly continuous on $[0,\infty)\times \DD$.

\item[(ii)] For every fixed $x\in \mathcal{D}$, $t\mapsto V(t,x)$
is differentiable in $[0,\infty)$  and
\begin{equation}\label{dV}
\frac{d V(t,x)}{d t}=\Gamma_t(x);
\end{equation}

\item[(iii)] For every fixed $x_0\in \mathcal{D}$, there exists $
t_0 >0$ such that $V(t,x_0)$ is invertible for every $t \in (0,
t_0 ]$;

\item [(iv)] For all fixed $x\in \mathcal{D}$ and $t,s \in
[0,\infty)$, the following identity holds:
\begin{equation}\label{id-V}
V(s,F_t(x))\Gamma_t(x)=V(t+s,x)-V(t,x);
\end{equation}
\end{itemize}
\end{lemma}

Concerning assertion (i), note that by Theorem~\ref{th-JC} every
semicocycle over a jointly continuous semigroup is jointly
continuous itself.

\begin{proof}
Assertions (i) and (ii) follow immediately.

Fix some $x_0\in \mathcal{D}$. By assertion (ii) we have
$$\lim_{t\to 0^+} \frac{1}{t}V(t,x_0)=\Gamma_0(x_0)=1_\mathcal{A}.$$
This means that if we choose $ t_0 >0$ sufficiently small,
$\frac{1}{t}V(t,x_0)$ is invertible, and hence $V(t,x_0)$ is also
invertible for all $0< t< t_0 $. Therefore (iii) holds.

Assertion (iv) follows from the chain rule. Indeed,
    for all $t,s \in [0,\infty)$,
    \begin{eqnarray*}%\label{id0}
    &&V(s,F_t(x))\Gamma_t(x)=\int_0^{s} \Gamma_r(F_t(x))
    \Gamma_t(x)dr  \nonumber \\
    &=&\int_0^{s} \Gamma_{t+r}(x)dr =\int_{t}^{t+s}
    \Gamma_{r}(x)dr
    \nonumber  \\
    &=&V(s+t,x)-V(t,x).
    \end{eqnarray*}
This completes the proof.
\end{proof}

We now present the main results of this section for semicocycles
consisting of smooth mappings.

\begin{theorem}\label{th-differentiability}
       Let a semigroup $\Ff=\left\{F_t\right\}_{t\ge0}\subset C(\DD)$ be generated by $f \in C(\DD, X)$. Assume that $\{\Gamma_t\}_{t \geq 0}\subset C^1(\DD, \A)$ is a semicocycle over $\Ff$, such that the family $\{{\Gamma_t}' [f] \}_{t\ge0}$ is jointly continuous on $[0,\infty)\times \DD$.
   Then for each $x\in \DD$, the function $t\mapsto \Gamma_t(x)$ is differentiable on $[0,\infty)$. Moreover, defining
\begin{equation}\label{B}
B(x)=\left.\frac{d}{dt}\Gamma_t(x)\right|_{t=0},
\end{equation}
we have ${B\in C(\mathcal{D},\A)}$ and $\Gamma_t(x)$ is the unique
solution to the evolution problem~\eqref{cauchy}.
\end{theorem}

\begin{proof}
Fix some $x_0\in \mathcal{D}$. Let us show that $t\mapsto
\Gamma_t(x_0)$ is differentiable in $[0,\infty)$. By assertion
(iv) of Lemma~\ref{lem-V}, for any~$t$ we have
$$V(s,F_t(x_0))\Gamma_t(x_0)=V(t+s,x_0)-V(t,x_0).$$
By assertion (iii) of Lemma~\ref{lem-V}, there exists $t_0>0$ such
that $V(t_0,x_0)$ is invertible. Hence $V(t_0,x)$ is invertible
for all $x$ in a neighborhood $U$ of~$x_0$. In particular, there
exists $t_1>0$ such that  $V(t_0,F_t(x_0))$ is invertible for $t<
t_1$. Thus we can write the above as
\[
\Gamma_t(x_0) =V(t_0,F_t(x_0))^{-1}[ V(t+t_0,x_0)-V(t,x_0)].
\]
First we note that $t\mapsto [ V(t+t_0,x_0)-V(t,x_0)]$ is
differentiable for all $t\ge 0$ by (ii) of Lemma~\ref{lem-V}. To
see the differentiability of $t\mapsto V(t_0,F_t(x_0))^{-1}$ for
$t< t_1$, it is enough to verify that $V(t_0,F_t(x_0))=
\int_{0}^{t_0} \Gamma_\tau(F_t(x_0))d\tau$ is differentiable.
Indeed, since the family $\{{\Gamma_t}' [f] \}_{t\ge0}$ is jointly
continuous, differentiation under the integral is valid by
Leibniz's rule (see \cite[8.11.2]{Die}). Moreover,
\begin{equation}\label{*}
\frac{d}{dt}V(t_0,F_t(x_0))= \int_{0}^{t_0} {\Gamma_\tau}'
(F_t(x_0))[f(F_t(x_0))]d\tau.
\end{equation}
We therefore conclude that $t\mapsto \Gamma_t(x_0)$ is
differentiable for $t\in[0,t_1)$.

To show the differentiability for all $t\ge0$, we fix any $s>0$
and write
\[
\Gamma_{t+s}(x_0) =\Gamma_s(F_t(x_0)) \Gamma_t(x_0).
\]
The last factor in the right hand-side is differentiable with
respect to $t<t_1$ while the first one is differentiable for all
$t\ge0$. Therefore we conclude that $t\mapsto \Gamma_t(x_0)$ is
differentiable in $[s,s+t_1)$; and since $s$ is arbitrary, we are
finished.

Now we find an explicit formula for the derivative
$B(x)=\left.\frac{d}{dt}\Gamma_t(x)\right|_{t=0}$. Differentiating
both sides of \eqref{id-V} with respect to $t$ by using \eqref{*}, and then setting
$t=0$, we obtain
\begin{equation}\label{B1}
  B(x)=V(s,x)^{-1}\left(\Gamma_{s}(x)-1_{\A}-      \int_{0}^{s}
{\Gamma_\tau}'(x)[f(x)]d\tau                   \right).
\end{equation}
This implies that $B$ is continuous on~$\mathcal{D}$. Moreover,  for any $t$,
\begin{eqnarray}\label{lim-2g}
&&\frac{d}{d t} \Gamma_t(x) = \lim_{s\to0^+}\frac1s
\left(\Gamma_{t+s}(x)-
\Gamma_t(x)\right) \nonumber \\
&=& \lim_{s\to0^+}\frac1s \left(\Gamma_s(F_t(x))-
1_{\A}\right)\Gamma_t(x) =B(F_t(x))\Gamma_t(x).
\end{eqnarray}
Therefore, for every $x\in \mathcal{D}$, the $\A$-valued mapping
$v(t,x)= \Gamma_t(x)$ is the unique solution of \eqref{cauchy}.
\end{proof}

\begin{remark}\label{rem}
Following the analogy with generation theory for semigroups, it is natural to call the mapping~$B$ defined by \eqref{B} {\it the generator} of the semicocycle $\{\Gamma_t\}_{t\ge0}$.
\end{remark}

Next we assume that a semicocycle $\left\{\Gamma_t\right\}_{t\ge0}\subset C(\DD,\A)$ is
differentiable with respect to $t$ and denote its generator by $B$, which is also assumed to be continuous on $\DD$. In this setting one asks:

{\it How can one estimate the growth of a differentiable semicocycle?}

This  problem is important because such growth estimates provide affirmative answers to questions concerning the stability of evolution problems \eqref{cauchy}. We establish uniform estimates for the growth of the norm $\|\Gamma_t(x)\|_\A$ of a semicocycle with respect to time. Some of these estimates are sharp for small $t$, while others relate to the asymptotic behavior as $t\rightarrow \infty$.

Since $\|\Gamma_t(x)\|_\A\rightarrow 1$ as $t\rightarrow 0$, we are interested in obtaining bounds of the form
\begin{equation}\label{K}
\left\|\Gamma_t\right\|_{\DD^*}:=\sup\limits_{x\in\DD^*}\|\Gamma_t(x)\|_\A
\leq e^{Kt},
\end{equation}
on suitable subsets
$\DD^* \subseteq \mathcal{D}$ and exponents $K$, which depend on $\DD^*$. Being solutions of evolution problems, differentiable semicocycles satisfy estimates of the form \eqref{estim-u}; hence, they are expected to satisfy \eqref{K}. The following assertion justifies this expectation, generalizing Theorem~2.3 in\cite{EJK}. 
We use the quantity $\mu(a), a \in \A$, as defined in assertion (ii) of Theorem~\ref{th-sol-cauchy_1}.

\begin{propo}\label{pro-exp}
    Let $\DD^*\subseteq \DD$ be an invariant set of a semigroup $\mathcal{F}=\{F_t\}_{t\ge0}\subset C(\DD)$.
    A semicocycle $\left\{\Gamma_t\right\}_{t\ge0}$ with generator $B$ satisfies \eqref{K} if and only if its generator $B$ satisfies
    $\mu(B(x))\leq K$ for all
    $x\in\DD^*$.
\end{propo}

\begin{proof}
    Suppose that $\left\{\Gamma_t\right\}_{t\ge0}$ satisfies \eqref{K}. Since $B$ is the derivative of
    $\Gamma_t(x)$ at $t=0$, we have
    \[
    \Gamma_t(x)=1_\A+ tB(x)+o(t) \quad \mbox{ as } t \to 0.
    \]
    Therefore, for $x\in \DD^*$,
    \begin{eqnarray*}
        \mu(B(x))&=&\lim\limits_{t\to 0^+} \frac{1}{t} \left[\left\|1_{\A}+tB(x)
        \right\|_\A -1\right]= \lim\limits_{t\to 0^+} \frac{1}{t} \left[\left\|\Gamma_t(x)
        \right\|_\A -1\right]\\ &\leq& \lim\limits_{t\to 0^+} \frac 1
        t(e^{Kt}-1)=K.
    \end{eqnarray*}

Conversely, let generator $B$ of a semicocycle satisfy $\mu(B(x))\leq K$ for all $x \in\DD^*$.  Then, since $\DD^*$ is $\Ff$-invariant, if $x\in\DD^*$, then  $F_t(x)\in \DD^*$ for all $t\geq 0$, so $\mu\left(B(F_t(x))\right)\leq K$, and the using of Theorem~\ref{th-sol-cauchy_1} completes the proof.
\end{proof}

We remark that there exist semicocycles which satisfy \eqref{K} on every
$\Ff$-invariant domain different from the whole $\DD$ but not
on $\DD$ (see Examples~2.1--2.2 in \cite{EJK}).

To formulate our next result recall that a point $x_0 \in \DD$ is called a fixed point of a semigroup $\mathcal{F} =\left\{F_t\right\}_{t\ge0}$ if $F_t(x_0)=x_0$ for all $t \geq 0$. A fixed point $x_0$ is said to be (globally) {\it{attractive}} if  $F_t(x)\to x_0$ as $t\to\infty$ for all $x\in \DD$. While in the one-dimensional case, for a semigroup of holomorphic self-mappings, which does not consist of elliptic automorphisms, the existence of a unique fixed point implies its attractivity, this is no longer true in higher-dimensional spaces, as shown by the simple example:  $F_t(x_1,x_2)=(e^{-t}x_1, e^{it} x_2)$. Moreover, in an infinite-dimensional space a semigroup may converge to a unique interior point  $x_0\in\DD$ but not uniformly on any neighborhood of $x_0$.

\vspace{2mm}

Notice that an inequality of the form $\|\Gamma_t(x)\|_\A\le Le^{Kt}$ is not sharp for small $t$ whenever $L>1$. On the other hand, it might happen that for large $t$, one can admit a smaller exponent at the expense of taking a larger prefactor $L$. We now present sharp (as $t\to\infty$) estimates for semicocycles over semigroups having a globally attractive fixed point.

The following result shows that the exponential rate of growth of a semicocycle over a semigroup with an attractive fixed point $x_0$ is determined by the rate of growth of the linear semigroup $\Gamma_t(x_0)=e^{tB_0}$, where $B_0=B(x_0)$. To formulate it, we denote the spectrum of an element $A$ of a Banach algebra by $\sigma(A)$, and the Lyapunov index of $A$  (see, for
example, \cite{Kr} or \cite{D-Sch1958}) by
\begin{equation}\label{kappa+}
  \kappa(A):=\max\{\Re \lambda: \lambda\in\sigma(A)\} =
  \limsup_{t\rightarrow \infty}\frac{1}{t} \log\|e^{tA}\|.
\end{equation}

\begin{theorem}\label{th-estim1}
    Let $\Ff=\{F_t\}_{t\ge0}\subset C(\DD)$ be a semigroup having a globally attractive fixed point $x_0\in\mathcal{D}$. Let $\left\{\Gamma_t\right\}_{t\ge0}\subset C(\DD,\A)$  be a semicocycle over $\Ff$ generated by $B \in C(\DD,\A)$.
    Let $\kappa=\kappa(B(x_0))$ be given by~\eqref{kappa+}. Then, for every $x\in \mathcal{D}$,
    $$\limsup_{t\rightarrow \infty} \frac{1}{t}\ln \|\Gamma_t(x)\|_\A \leq \kappa.$$
    Moreover, if the convergence $F_t(x)\to x_0$ is uniform on all sets strictly inside $\mathcal{D}$,
    then for every $\DD^*$ which is strictly inside $\mathcal{D}$ and $\epsilon>0$ there exists an
    $L=L(\mathcal{D}^*,\epsilon)$ such that
    \[
    \|\Gamma_t(x)\|_\A\le L e^{t(\kappa+\epsilon)},\quad x\in\mathcal{D}^*.
    \]
\end{theorem}

\begin{proof}
    Fixing any $\epsilon>0$, by \eqref{kappa+}, we can find $C$
    such that
    $$\|\Gamma_t(x_0)\|=\|e^{tB_0}\|\leq C e^{(\kappa+\epsilon)t}.$$
    Fixing $x \in \DD$ and defining $A(t)=\Gamma_t(x)- \Gamma_t(x_0)$, we have
    \[
    \frac{d}{dt}A(t)=B(F_t(x))\Gamma_t(x)-B(x_0)\Gamma_t(x_0)=[B(F_t(x))-B_0]\Gamma_t(x)+B_0A(t).
    \]
    Solving this linear non-homogenous equation for $A(t)$, we get
    $$A(t) = \int_0^t e^{(t-s)B_0}[B(F_s(x))-B_0]\Gamma_s(x)ds, $$
    hence
    \begin{eqnarray}\label{ki}
    &&\|\Gamma_t(x)\|_\A =\|\Gamma_t(x_0)+A(t)\|_\A  \nonumber\\
    &\leq& \|\Gamma_t(x_0)\|_\A+ \int_0^t \|e^{(t-s)B_0}\|_\A \cdot
    \| B(F_s(x))-B_0\|_\A\cdot \|\Gamma_s(x)\|_\A ds \nonumber\\
    &\leq& C e^{(\kappa+\epsilon)t}+ C\int_0^t  e^{(\kappa+\epsilon)(t-s)}
    \| B(F_s(x))-B_0\| _\A\cdot\|\Gamma_s(x)\|_\A ds. \hspace{10mm}
    \end{eqnarray}
    Using Gr\"onwall's inequality (see {\it} e.g.
    \cite[Ch.2, Corollary 2.2]{Kr}), \eqref{ki} implies
    $$\|\Gamma_t(x)\|_\A\leq Ce^{(\kappa+\epsilon) t + C\int_0^t \| B(F_t(x))-B_0\|ds},$$
    that is,
    \begin{equation}\label{ki1}
    \log \|\Gamma_t(x)\|_\A \leq \log C+ (\kappa+\epsilon) t +
    C\int_0^t \| B(F_s(x))-B_0\|_\A ds.
    \end{equation}
    According to the continuity of $B$ assumed at the beginning of the section, $\lim\limits_{t\rightarrow \infty} \|
    B(F_t(x))-B_0\|_\A=0$, hence
    $$\lim_{t\rightarrow \infty}\frac{1}{t}\int_0^t \| B(F_s(x))-B_0\|_\A ds=0.$$
    Therefore \eqref{ki1} implies
    $$\limsup_{t\rightarrow \infty} \frac{1}{t}\log \|\Gamma_t(x)\|_\A \leq \kappa+\epsilon,  $$
    and since $\epsilon>0$ was arbitrary, the first assertion follows.

    To prove the second assertion, we first find a neighborhood $U$ of
    $x_0$ such that ${\|B(x)-B(x_0)\|_\A<\epsilon}$ for all $x\in U$.
    Further, for every $\mathcal{D}^*$ which is strictly inside $\mathcal{D}$ there is $T>0$ such that $F_t(x)\in U$ for all $x\in\mathcal{D}^*$ and $t>T$. Then \eqref{ki1} can be rewritten
    in the form
    \[
    \log \|\Gamma_t(x_1)\|_\A \leq \log C+ (\kappa+\epsilon) t +
    C\int_0^T \| B(F_s(x_1))-B_0\|_\A ds +C\epsilon(t-T),
    \]
    which implies the result.
\end{proof}

\section{Holomorphic semicocycles}\label{Hol-semico}

\setcounter{equation}{0}

Let $\DD$ be a bounded domain  in a complex Banach space.
In this section we study holomorphic semicocycles $\{\Gamma_t\}_{t\ge 0} \subset \Hol(\DD,\A)$ over $\Ff \subset \Hol(\DD)$.
For such semicocycles, we intend to sharpen the results of the previous section concerning differentiability and growth estimates.

Our aim is to show the one-to-one correspondence between generators and UJC semicocycles: every $B\in \Hol(\DD,\A)$ generates a UJC semicocycle $\{\Gamma_t\}_{t\ge 0} \subset \Hol(\DD,\A)$, and every UJC holomorphic semicocycle is generated by some $B\in \Hol(\DD,\A)$. Moreover, a semicocycle is $T$-continuous if and only if its generator is bounded on sets strictly inside~$\DD$.

First, we mention that semicocycles in Examples~\ref{examp12+}--\ref{examp12} are holomorphic, so  there are holomorphic semicocycles which are not UJC and UJC holomorphic semicocycles which are not $T$-continuous.

Throughout the section, we assume that $\Ff$ is $T$-continuous. Consider now the unique solution to the evolution problem
\begin{equation}\label{cauchy-a}
\left\{
\begin{array}{l}
\displaystyle \frac{d v(t,x)}{d t} =B(F_t(x))v(t,x) \vspace{2mm} \\
v(0,x)=1_\A,
\end{array}%
\right.
\end{equation}
where $B\in \Hol(\mathcal{D},\A)$ and  $\Ff \subset \Hol(\DD)$. By Theorem~\ref{th-sol-cauchy}, the family $\{\Gamma_t:=v(t,\cdot)\}_{t\ge0}$ is a semicocycle over $\Ff$. Using assertion (i) of Theorem~\ref{th-sol-cauchy_1} and the  implicit function theorem in Banach spaces, one can easily see that this semicocycle is holomorphic.

\begin{theorem}\label{th-differentiability-a}
	Let $\Ff=\left\{F_t\right\}_{t\ge0}\subset\Hol(\DD)$ be a $T$-continuous semigroup on a bounded domain $\DD$. Then a semicocycle $\{\Gamma_t\}_{t \geq 0}\subset \Hol(\DD, \A)$ over $\Ff$ is generated by a mapping $B\in\Hol(\DD,\A)$ if and only if it is uniformly jointly continuous.
\end{theorem}

\begin{proof}
We first show that a semicocycle generated by (\ref{cauchy-a}) is UJC. By Lemma \ref{lem-LUC}, it suffices to show it is UJC
at $t=0$. Fix $x_0\in \DD$.
The holomorphic mapping $B$ is bounded in a neighborhood $U\subset\DD$ of $x_0$, say $\|B(x)\|_\A\le M$ for $x\in U$.
Let $U^*$ be a neighborhood of $x_0$ which lies strictly inside $U$.
 By $T$-continuity of $\Ff$, for every $\epsilon>0$ there is $\delta>0$ such that $\|F_t(x)-x\|_X<\epsilon$ for all $x\in U^*$ whenever $t\in[0,\delta)$ . Hence there is  $\delta>0$ such that $F_t(x)\in U$ for $(t,x)\in [0,\delta)\times U^*$. Therefore $\|B(F_{s}(x))\|_\A\le M$ and by assertion (ii) of Theorem~\ref{th-sol-cauchy_1}, $\|\Gamma_s(x)\|_\A \le e^{sM}$ for all $s\le t<\delta$ and $x\in U^*$. Thus by the semicocycle property and by assertion (i) of Theorem~\ref{th-sol-cauchy_1},
$$\left\|\Gamma_{t}(x) -1\right\|=\left\|\int_0^t \frac{d}{ds}\Gamma_{s}(x)ds \right\|=\left\|\int_0^t B(F_{s}(x))  \Gamma_{s}(x)  ds\right\|  \le \left(e^{tM}-1\right)M$$
for $(t,x)\in [0,\delta)\times U^*$, implies the uniform joint continuity.

Conversely, let $\{\Gamma_t\}_{t\ge0}$ be a UJC holomorphic semicocycle. Then by Lemma~\ref{lem-T-hol}, the family $\{{\Gamma_t}'\}_{t\ge0}$ is also UJC, hence JC by Proposition~\ref{th-cont}. Then applying Theorem~\ref{th-differentiability}, we conclude that the semicocycle is differentiable. Finally, by formula~\eqref{B1}, $B\in \Hol(\DD,\A)$.
\end{proof}

\begin{remark}
  In the case where $X$ is finite-dimensional, every semigroup $\Ff\in\Hol(\DD)$ is JC on $[0,\infty)\times \A$ by Theorem~\ref{th-Ch-M-B}, and is thus $T$-continuous by Proposition~\ref{th-cont}(iii), and every holomorphic semicocycle $\{\Gamma_t\}_{t\ge0}$ over $\Ff$ is JC  by Theorem~\ref{th-JC}, hence it is UJC. Thus, applying as above  Lemma~\ref{lem-T-hol} and  Theorem~\ref{th-differentiability} we see that

  Any holomorphic semicocycle on a bounded domain in a finite-dimensional space is a solution of an evolution problem~\eqref{cauchy-a}.
\end{remark}
Another surprising fact can be seen analysing the arguments leading to Theorem~\ref{th-differentiability-a}.
\begin{corol}
  Let $\{\Gamma_t\}_{t\ge0}$ be a semicocycle over a $T$-continuous semigroup. If the family $\{{\Gamma_t}'\}_{t\ge0}$ is jointly continuous, it is uniformly jointly continuous.
\end{corol}

The previous theorem provides conditions for a holomorphic semicocycle to be a solution of evolution problem \eqref{cauchy-a}. Recall that the differentiability of a holomorphic semigroup and the boundedness of its generator is tightly connected to $T$-continuity of a semigroup (see Theorem~\ref{th-RS}). It follows from Theorem~\ref{th-differentiability-a} that for differentiability of a semicocycle its $T$-continuity is not needed. At the same time, the next result shows a partial similarity with the semigroup case.

\begin{theorem}\label{th-eq-Pr-Cauchy}
Assume that $\left\{\Gamma_t\right\}_{t\ge0}\subset\Hol(\DD,\A)$ is a semicocycle over a $T$-continuous semigroup $\mathcal{F}=\left\{F_t\right\}_{t\ge0}\subset\Hol(\DD)$. Then $\left\{\Gamma_t\right\}_{t\ge0}$ is $T$-continuous if and only
if it is generated by a mapping ${B\in\Hol(\mathcal{D},\A)}$, which is bounded on every domain $\DD^*$ strictly inside $\mathcal{D}$.
\end{theorem}

\begin{proof}
Assume that $\left\{\Gamma_t\right\}_{t\ge0}$ is $T$-continuous. Then by Lemma~\ref{lem-T-hol}, the family $\left\{{\Gamma_t}'\right\}_{t\ge0}$ is also $T$-continuous; and hence, by Proposition~\ref{th-cont}, it is jointly continuous. Therefore for every $t\ge0$, the mapping $V(t,x)$ defined by \eqref{V-func} is holomorphic and Theorem~\ref{th-differentiability} can be applied (equivalently, one uses Theorem~\ref{th-differentiability-a}). Hence $\Gamma_t(x)$ is the unique solution of \eqref{cauchy-a}.

Choose any two domains $\mathcal{D}_1$ and $\mathcal{D}_2$ such that $\mathcal{D}_1$ is strictly inside $\mathcal{D}_2$ and $\mathcal{D}_2$ is strictly inside $\mathcal{D}$. We can choose $t_1$ so that
    \begin{equation}\label{ee0}
x\in\mathcal{D}_2, \,  t\in [0,t_1]\quad \Rightarrow\quad \left\|
\Gamma_t(x)-1_{\A}\right\|_\A \leq 1 \quad \Rightarrow\quad
\left\|\Gamma_t(x)\right\|_\A\leq 2.
    \end{equation}

Since the semigroup $\mathcal{F}$ is $T$-continuous, it follows from Theorem~\ref{th-RS} that it is generated by some mapping
$f\in\Hol(\mathcal{D},X)$, which is bounded on each domain strictly inside $\mathcal{D}$. Therefore there is $m_0$ so that
\begin{equation}\label{f-est}
\|f(x)\|_X\leq m_0 \mbox{ for all } x\in\mathcal{D}_2.
\end{equation}
By \eqref{ee0} and Cauchy's inequality, we have
$$x\in \mathcal{D}_1,\;t\in [0,t_1]\;\;\Rightarrow\;\;\|\Gamma_{t}'(x)\|_{\A} \leq m_1.$$
This implies that the function $V(t,x)$ defined by \eqref{V-func} satisfies
\begin{equation}\label{ee3}
x\in \mathcal{D}_1,\; t\in [0,t_1]\; \;\Rightarrow\;\;\left\|V'(t,x)\right\|_{L(X,\A)}\leq \int_0^{t}\|\Gamma_{s}'(x)\|_\A ds \leq m_1 t.
\end{equation}

To proceed we show now that there exists $\tau  >0$ such that
\begin{equation}\label{est-invV}
\left\|\left[ V(t,x)\right]^{-1}\right\|_\A \leq \frac{2}{t}
\end{equation}
for all $0<t< \tau$ and $x\in \mathcal{D}_2$. Indeed, since $\left\{\Gamma_t\right\}_{t\ge0}$ is $T$-continuous, the convergence $\frac{1}{t}V(t,x)\to 1_\A$ as $t \to 0$, is uniform on~$\mathcal{D}_2.$ So there exists $\tau  >0$ such that the inequality $\displaystyle \left\| \frac{1}{t}V(t,x)
-1_\A \right\|_\A \leq \frac{1}{2}\,$ holds for all $t<\tau$ and $x\in\mathcal{D}_2$. Thus
\[
\left\|\left[ \frac{1}{t}V(t,x)\right]^{-1} \right\|_\A \leq
\frac{1}{1-\left\| \frac{1}{t}V(t,x) -1_\A\right\|_\A}\leq 2,
\]
which implies \eqref{est-invV}.

Let us now turn to \eqref{B1}. This formula shows that in our setting $B\in\Hol(\DD,\A)$. Further, assuming $x\in\mathcal{D}_1$ and using \eqref{ee0}--\eqref{est-invV}, it implies that, for $s<\min(t_1,\tau)$,
\begin{eqnarray*}
&&\|B(x)\|_\A\leq \left\| [V(s,x)]^{-1}\right\|_\A \cdot
\left\|\left[\Gamma_{s}(x)-1_\A -V'(s,x)[f(x)]\right]\right\|_\A \\
&\leq& \left\| [V(s,x)]^{-1}\right\|_\A \cdot \left[\|[\Gamma_
{s}(x)\|_\A
+ 1+ \left\|V' (s,x)\right\|_{L(X,\A)} \cdot \|f(x)\|_X\right] \\
&\leq& \frac{2}{s} [2+1+m_1 s\cdot m_0].
\end{eqnarray*}

We have therefore shown that $B$ is bounded on $\mathcal{D}_1$.
Since $\mathcal{D}_1$ is an arbitrary domain strictly inside $\mathcal{D}$, we obtain that $B$ is bounded on every domain strictly inside $\mathcal{D}$.

Conversely, assume that $\Gamma_t(x)$ satisfies the evolution
problem \eqref{cauchy-a}, where ${B\in\Hol(\mathcal{D},\A)}$ is bounded on every domain strictly inside $\mathcal{D}$. Then  $\{\Gamma_t\}_{t\ge0}$ is a semicocycle over $\{F_t\}_{t \ge 0}$ by Theorem \ref{th-sol-cauchy}.

Since $\{F_t\}_{t \ge 0}$ is $T$-continuous, then for any two domains $\mathcal{D}_2$  strictly inside $\mathcal{D}$ and $\mathcal{D}_1$ strictly inside $\mathcal{D}_2$ and for any ${0<\varepsilon<{\rm dist}\{\partial\mathcal{D}_1,\partial\mathcal{D}_2 \}}$ there is $t^*>0$ such that
for all $s\in[0,t^*]$ we have ${\left\|F_s(x)-x\right\|_X<\varepsilon}$ whenever $x\in\mathcal{D}_1$.
Since $B$ is bounded on every domain strictly inside $\mathcal{D}$, there exists $M>0$ such that $\|B(x)\|_\A\leq M$ on the set
 $\{x: {\rm dist}(x,\mathcal{D}_1)\leq \epsilon \}$. Thus $\|B(F_s(x)\|_\A \le M$
whenever $x \in \mathcal{D}_1$. In particular, this implies by Theorem~\ref{th-sol-cauchy_1}(ii) that $\|\Gamma_s(x)\|_\A\le
e^{Ms}$ for all $s\in[0,t^*]$ and $x\in\overline{\mathcal{D}_1}$.
In addition, $\Gamma_t(x)$ satisfies the integral
equation~\eqref{integral-form}. Therefore, for $x\in \mathcal{D}_1,0\leq t\leq t^*$,
\begin{eqnarray*}
\|\Gamma_t(x)-1_\A\|_\A &\leq&
\int_0^t\|B(F_s(x))\|_\A \cdot \| \Gamma_s(x)\|_\A ds \\
&\le& \int_0^t M e^{Ms}ds =e^{Mt}-1.
\end{eqnarray*}
Thus, $\left\{\Gamma_t\right\}_{t\ge0}$ is $T$-continuous.
\end{proof}

Returning to Example~\ref{examp12} of a semicocycle which is not
$T$-continuous, we see that its generator, given by $B(x)=
\sum\limits_{k=1}^\infty k\left(\frac{x_k}{\rho}\right)^k$, is
unbounded on the ball of radius $\rho$.

It turns out that for holomorphic semicocycles the existence of bounds of the form~\eqref{K}, that is, $\left\|\Gamma_t\right\|_{\DD^*}\leq e^{Kt}$, on suitable subsets $\DD^* \subseteq \mathcal{D}$ and exponents $K=K(\DD)$ is intimately connected to $T$-continuity of the semicocycle.

\begin{propo}\label{propo-growth}
    Let $\DD \subset X$ be a bounded domain. Let $\left\{\Gamma_t\right\}_{t\ge0}\subset\Hol(\DD,\A)$ be a semicocycle over a $T$-continuous semigroup $\mathcal{F}\subset \Hol(\DD)$. If there exists $K$ such that $\left\|\Gamma_t\right\|_{\DD}\leq e^{Kt}$, then $\{\Gamma_t\}_{t\ge0}$ is $T$-continuous.
\end{propo}
The proof of this proposition follows directly from the proof of Theorem~\ref{th-Tcont-B} below.

Regarding the assertion converse to Proposition~\ref{propo-growth}, it is certainly not true that $T$-continuity implies that $\left\|\Gamma_t\right\|_{\DD}\leq e^{Kt}$, as is shown by Example~2.1 in \cite{EJK}.
Nevertheless, under an additional assumption on the semigroup behavior, $T$-continuity of semicocycle is equivalent to the property that $\left\|\Gamma_t\right\|_{\DD^*}\leq e^{Kt}$ on each $\mathcal{D}^*$ strictly inside $\DD$. More precisely,

\begin{theorem}\label{th-Tcont-B}
   Let $\DD \subset X$ be a bounded domain. Let $\left\{\Gamma_t\right\}_{t\ge0}\subset\Hol(\DD,\A)$ be a semicocycle over a $T$-continuous semigroup $\mathcal{F}\subset\Hol(\DD)$.
  Assume that for every set $\DD_0$ strictly inside $\DD$ there is an $\Ff$-invariant domain $\DD^*$ strictly inside $\DD$ such that $\DD_0 \subset \DD^*$. Then $\{\Gamma_t\}_{t\ge0}$ is $T$-continuous if and only if for every subset $\DD^*$ strictly inside $\DD$ there is $K$ such that $\left\|\Gamma_t\right\|_{\DD^*}\leq e^{Kt}$.
\end{theorem}

The assumption on the semigroup behavior in this theorem holds whenever a semigroup $\Ff$ has a globally attractive fixed point (see for example, \cite{G-R, R-S1}).

\begin{proof}
Suppose that $\{\Gamma_t\}_{t\ge0}$ is $T$-continuous. By Theorem~\ref{th-eq-Pr-Cauchy} it is generated by a holomorphic mapping $B$. Let $\DD_0$ be any set strictly inside $\DD$. By our assumption there is an $\Ff$-invariant domain $\DD^*$ strictly inside $\DD$. By Theorem~\ref{th-eq-Pr-Cauchy}, the supremum $K:=\sup\limits_{x\in\DD^*}\|B(F_s(x))\|_\A$ is finite. Therefore $$\sup\limits_{x\in\DD_0} \mu(B(F_s(x))) \le K.$$ So, $\left\|\Gamma_t\right\|_{\DD^*}\leq e^{Kt}$  by assertion (ii) of Theorem~\ref{th-sol-cauchy_1}.

Conversely, fix an arbitrary set $\DD_0$ strictly inside $\DD$. Suppose that there is an $\Ff$-invariant domain $\DD^*$ strictly inside $\DD$ such that $\DD_0 \subset \DD^*$ and $K=K(\DD^*)$ such that $\|\Gamma_t(x)\|_\A \leq e^{Kt}$ for all $x\in\DD^*$.

Denote $\hat\Gamma_t(x):=e^{-Kt}\Gamma_t(x)$, $t \geq 0$. Then $\|\hat\Gamma_t(x)\|_{\A} \le1$ for all $x\in\DD^*$. It is easy to see that the family $\left\{\hat\Gamma_t \right\}_{t\ge0}$ is a semicocycle.

Then, by Proposition~\ref{th-sg}, the family $\mathcal{\widetilde F}$ defined
by ${\widetilde{F}_t(x,a):= \left(F_t(x),\hat\Gamma_t (x) a
\right)}$ forms a semigroup on $\DD\times\A$. Now we
verify that the bounded domain $\widetilde{\DD}:=\left\{(x,a):\ x\in\DD^*,\
\|a\|_\A<1 \right\}$ is $\widetilde{\mathcal{F}}$-invariant.
Indeed, for every point $(x,a)\in\widetilde{\DD}$ we have
$F_t(x)\in\DD^*$ and
\begin{eqnarray*}
&&\left\|\hat\Gamma_t(x)a \right\|_\A \le \left\| \hat\Gamma_t(x)
\right\|_\A\cdot \|a\|_\A <1,
\end{eqnarray*}
that is, $\widetilde{F}_t(x,a)\in\widetilde{\DD}$. So, $\widetilde{\Ff}$ is a
semigroup on $\widetilde{\DD}$. It is generated by the mapping
\[
\widetilde f(x,a)=\left(f(x),\left(B(x)-K\cdot 1_\A\right)a
\right),
\]
where $f$ and $B$ are the generators of $\Ff$ and
$\{\Gamma_t\}_{t\ge0}$, respectively (cf.
Corollary~\ref{cor-gen}). Since $\widetilde f$ generates a semigroup of holomorphic self-mappings on the
bounded domain $\widetilde{\DD}$, it is holomorphic and is bounded strictly inside $\widetilde{\DD}$; see Theorem~\ref{th-RS}. Hence $B\in\Hol(\DD,\A)$ and is bounded on $D_0$ (which is strictly inside $\DD^*$).
Thus, it follows from Theorem~\ref{th-eq-Pr-Cauchy} that
$\{\Gamma_t\}_{t\ge0}$ is $T$-continuous.
\end{proof}

%%%%%%%%%%%%%%%%%%%%%%%%%%%%%%%

%%%%%%%%%%%%%%%%%%%%%%%%%%%%%%%

\end{document}